\documentclass[11pt,english]{amsart}

\usepackage{amsmath}
\usepackage{dsfont}
\usepackage{amsfonts}
\usepackage{amsthm}
\usepackage{enumerate}
\usepackage{graphicx}
\usepackage{overpic}
\usepackage{a4wide}
\usepackage{amssymb} 
\usepackage{pictex}
\usepackage{rotating}  
\usepackage{xcolor}
\usepackage{etex}
\usepackage{float}
\usepackage[utf8]{inputenc}

\usepackage[all]{xy}

\numberwithin{equation}{section}

\newtheorem{theorem}{Theorem}[section]

\newtheorem{proposition}[theorem]{Proposition}
\newtheorem{lemma}[theorem]{Lemma}

\newtheorem{Definition}[theorem]{Definition}

\newenvironment{definition}{\begin{Definition}\rm}{\end{Definition}}
\newtheorem{Remark}[theorem]{Remark}
\newenvironment{remark}{\begin{Remark}\rm}{\end{Remark}}
\newtheorem{RHproblem}[theorem]{RH problem}

\newtheorem{Example}[theorem]{Example}

\theoremstyle{definition}

\newcommand{\indicatrice}[1]{\mathds{1}_{#1}}

\newcommand{\N}{\mathbb{N}}

\def\adots{\mathinner{\mkern2mu\raise 1pt\hbox{.}\mkern 3mu\raise 
3pt\hbox{.}\mkern1mu\raise 5pt\hbox{{.}}}}

\renewcommand{\tilde}{\widetilde}

\usepackage[bookmarksopen, naturalnames]{hyperref}

\makeatother

\begin{document}
\title[Frequent universality criterion and densities]{Frequent universality criterion and densities}
\author{R. Ernst, A. Mouze }
\thanks{The authors were partly supported by the grant ANR-17-CE40-0021 of the French National Research Agency ANR (project Front)}
\address{Romuald Ernst, LMPA, Centre Universitaire de la Mi-Voix, Maison de la Recherche Blaise-Pascal, 50 rue Ferdinand Buisson, BP 699, 62228 Calais Cedex}
\email{ernst.r@math.cnrs.fr}
\address{Augustin Mouze, Laboratoire Paul Painlev\'e, UMR 8524, 
Cit\'e Scientifique, 59650 Villeneuve d'Ascq, France, Current address: \'Ecole Centrale de
Lille, Cit\'e Scientifique, CS20048, 59651 Villeneuve d'Ascq cedex}
\email{Augustin.Mouze@math.univ-lille1.fr}


\begin{abstract} We improve a recent result by giving the optimal conclusion possible both to the frequent universality criterion and the frequent hypercyclicity criterion 
using the notion of $A$-densities, where $A$ refers to some weighted densities sharper than the natural lower density. Moreover we construct an operator which is logarithmically-frequently hypercyclic but not frequently hypercyclic. 
\end{abstract}

\maketitle

\section{Introduction and notations} We denote by $\mathbb{N}$ the set of positive integers. 
Let $(\alpha_k)$ be a non-negative sequence with $\sum_{k\geq 1}\alpha_k=+\infty.$ Let us consider 
the associated admissible matrix $A=(\alpha_{n,k})$ given by 
\[
\alpha_{n,k}=\left\{\begin{array}{l}\alpha_k/(\sum_{j=1}^n\alpha_j)\hbox{ for }1\leq k\leq n,\\
0\hbox{ otherwise.}\end{array}\right.
\] 
We know that every regular summability matrix $A$, so any admissible matrix, 
gives a density $\underline{d}_A$ on subsets of $\mathbb{N}$, called lower $A$-density \cite{Freedman}. 

\begin{definition}\label{defdens} {\rm For a regular matrix $A=(\alpha_{n,k})$ with non-negative coefficients 
and a set $E\subset \mathbb{N},$ the lower $A$-density 
of $E,$ denoted by $\underline{d}_A(E),$ is defined as follows 
$$\underline{d}_A(E)=\liminf_{n}
\left(\sum_{k=1}^{+\infty}\alpha_{n,k}\indicatrice{E}(k)\right),$$ 
and the associated upper $A$-density, denoted by $\overline{d}_A(E),$ is given by the equality 
$\overline{d}_A(E)=1-\underline{d}_A(\mathbb{N}\setminus E).$
}
\end{definition} 

Moreover it is well-known  \cite{Freedman} that the upper $A$-density of any set $E\subset\mathbb{N}$ is given by 
$\overline{d}_A(E)=\limsup_{n}\left(\sum_{k=1}^{+\infty}\alpha_{n,k}\indicatrice{E}(k)\right)$.\\

Let $X,Y$ be Fr\'echet spaces. In the present paper, we are interested in the 
universality of sequences of operators $(T_n),$ 
$T_n:X\rightarrow Y,$ in the following sense: a sequence $(T_n)$ 
is said to be \textit{universal} if there exists $x\in X$ such that the set 
${\{T_n x\ :\ n\in\mathbb{N}\}}$ is dense in $Y$. Such a vector $x$ is called 
an universal vector for $(T_n)$. When the sequence $(T_n)$ is given by the iterates 
of a single operator $T,$ i.e. $(T_n)=(T^n)$ and $Y=X$, the notion of universality reduces to the well-known one of \textit{hypercyclicity}, which is 
a central notion in linear dynamics.  
Now the following definition extends that of frequent universality and quantifies how often the orbit of an universal vector visits 
every non-empty open set. For any $x\in X$ and any subset $U\subset Y$, we set $N(x,U):=\{n\in\mathbb{N}: T_n x\in U\}.$ 

\begin{definition}\label{defiA}{\rm Let $A=\left(\alpha_k/\sum_{j=1}^n\alpha_j\right)$ be an admissible matrix. 
A sequence of operators $(T_n)$, $T_n:X\rightarrow Y,$ is called \textit{$A$-frequently universal} if 
there exists $x\in X$ such that for any non-empty open set $U\subset X,$ the set $N(x,U)$ has positive lower $A$-density.
}
\end{definition}  

Notice that if $\alpha_k=1,$ $k=1,2,\dots,$ the matrix $A$ corresponds to the Ces\`aro matrix, the lower $A$-density $\underline{d}_A$ coincides 
with the natural lower density $\underline{d}$ and we recover the notion 
of \textit{frequent universality} (or \textit{frequent hypercyclicity} in the case of sequences $(T^n)$).  
Bonilla and Grosse-Erdmann have derived a sufficient condition for a sequence of operators to be frequently universal \cite{Bongro}. Their criterion extends the 
frequent hypercyclicity criterion given by Bayart and Grivaux \cite{Baygrifrequentlyhcop}.

\begin{theorem} {\rm (Frequent Universality Criterion)} Let $X$ be a Fr\'echet space, $Y$ a separable Fr\'echet space and 
$T_n:X\rightarrow Y,$ $n\in\mathbb{N},$ continuous mappings. Suppose that there are a dense 
subset $Y_0$ of $Y$ and mappings $S_n:Y_0\rightarrow X,$ $n\in\mathbb{N},$ such that:
\begin{enumerate}
\item $\sum_{n=1}^{k}T_{k}S_{k-n} y$ converges unconditionally in $Y,$ uniformly in $k\in\mathbb{N},$ for all 
$y\in Y_0$;
\item $\sum_{n=1}^{+\infty}T_{k}S_{k+n} y$ converges unconditionally in $Y,$ uniformly in $k\in\mathbb{N},$ for all 
$y\in Y_0$;
\item $\sum_{n=1}^{+\infty}S_n y$ converges unconditionally in $X,$ for all 
$y\in Y_0$;
\item $T_nS_ny\rightarrow y$ for all $y\in Y_0.$ 
\end{enumerate}
Then the sequence $(T_n)$ is frequently universal. 
\end{theorem} 
 
It is well-known that there exist frequently universal sequences of operators which do not satisfy this criterion \cite{Baygriinv}. A natural question arises: 
if the sequence $(T_n)$ fulfills the hypotheses of the criterion, for which admissible matrices $A$ can one conclude that $(T_n)$ is $A$-frequently universal? Before answering, let us introduce some useful notations. 
We denote by $\log^{(s)}$ the iterated logarithmic function $\log\circ\log\circ\dots\circ\log$ 
where $\log$ appears $s$ times. In the sequel, we shall need the following admissible matrices:
\begin{enumerate}
\item $A_r=(e^{k^r}/\sum_{j=1}^n e^{j^r})$ for $0\leq r\leq 1$;
\item $\tilde{D}_s=(\alpha_k/\sum_{j=1}^n\alpha_k)$ given by the coefficients 
$\alpha_k=e^{k/(\log^{(s)}(k))}$ for $k$ large enough and $s\geq 2$;
\item $\tilde{B}_s=(\alpha_k/\sum_{j=1}^n\alpha_k)$ given by the coefficients 
$\alpha_k=e^{k/(\log(k)\log^{(s)}(k))}$ for $k$ large enough and $s\geq 2$;
\item $B_r=(\alpha_k/\sum_{j=1}^n\alpha_k)$ given by the coefficients 
$\alpha_k=e^{k/\log^r(k)}$ for $k$ large enough and $r\geq 1$; 
\item the matrix $L=(k^{-1}/\sum_{j=1}^nj^{-1})$ associated to the logarithmic 
density $\underline{d}_{log}$.
\end{enumerate}
Lemma 2.8 of \cite{ErMo} ensures that we have, for any $2\leq s\leq s'$, $1\leq t\leq t'$, $0<r<r'<1$ and for all subset 
$E\subset\mathbb{N},$
\[
{\bf\underline{d}_{A_1}(E)}\leq \underline{d}_{\tilde{D}_{s'}}(E)\leq 
\underline{d}_{\tilde{D}_s}(E)\leq {\bf \underline{d}_{B_1}(E)}\leq \underline{d}_{\tilde{B}_{s'}}(E)\leq 
\underline{d}_{\tilde{B}_s}(E)\]
\[\underline{d}_{\tilde{B}_{s}}(E)\leq\underline{d}_{B_t}(E)\leq \underline{d}_{B_{t'}}(E)
\leq \underline{d}_{A_{r'}}(E)\leq 
\underline{d}_{A_r}(E)\leq {\bf \underline{d}(E)}\leq 
\underline{d}_{log}(E).
\]
Recently the authors 
have showed that under the assumptions of the frequent universality criterion a sequence $(T_n)$ is automatically $\tilde{B}_s$-frequently universal 
for any positive integer $s\geq 1$ \cite{ErMo}. 
Actually the statement is written in the context of $\tilde{B}_s$-frequent hypercyclicity, but it is 
easy to check that the proof works along the same lines 
in the case of $\tilde{B}_s$-frequent universality. First 
we improve this result by showing that a sequence of operators which satisfies 
the frequent universality criterion is necessarily $B_1$-frequently universal. Then a 
technical modification of the proof allows us 
to show that such a sequence of operators is necessarily $\tilde{D}_s$-frequently universal for any $s\in\mathbb{N}$. 
On the other hand, we establish that an operator $T:X\rightarrow X$ cannot be $A_1$-frequently hypercyclic. Notice that this result was already proved whenever $X$ was a Banach space 
\cite[Proposition 3.7]{ErMo}.
Based on these results, this article determines determines exactly what quantifies the frequent universality criterion in terms of weighted densities of the return sets. The proof of this result essentially uses combinatorial arguments. 
Furthermore, in \cite{ErMo} the authors exhibited a frequently universal (hypercyclic) operator which 
is not $A_r$-frequently hypercyclic for any $0\leq r<1$ (in particular it does not 
satisfy the above criterion). In the opposite way, using 
similar ideas we build an operator which is log-frequently 
hypercyclic (i.e. $L$-frequently hypercyclic) but not frequently hypercyclic. Hence one can find 
hypercyclic operators whose orbit of some universal vectors visits rather often, in the sense of suitable lower $A$-densities, every non-empty set 
but not sufficiently to be frequently hypercyclic. As far as we know, 
it is the first example in this direction using lower densities.\\

The paper is organized as follows: in Sections \ref{cons_spec_sequ} and \ref{cons_spec_sequ2} we construct 
specific sequences of integers which will allow us to establish that the frequent universality 
criterion gives a stronger result. Then we will show that this new result is the best possible. 
Finally in Section \ref{fhc_op},
we exhibit an example, inspired by \cite{Bayru}, of an operator which is $L$-frequently hypercyclic but not frequently hypercyclic.      

\section{Construction of a specific sequence}\label{cons_spec_sequ} A careful examination of the proofs of the frequent universality 
criterion shows that it suffices to find a sequence $(\delta_k)$ of integers and $(n_k)$ an increasing sequence of integers such that 
\[
\vert n_k-n_l\vert\geq \delta_k+\delta_l\hbox{ whenever } k\ne l\hbox{ and for any }p\geq 1,\  \underline{d}\left(\{n_k :\delta_k=p\}\right)>0.
\]

We refer the reader to \cite[Lemma 6.19 and Theorem 6.18]{Bay} and \cite{Bongro}. 
An easy modification of the proof of the criterion allows to obtain the $A$-frequent universality provided that 
$\underline{d}_A\left(\{n_k :\delta_k=p\}\right)>0$ (see \cite{ErMo}). In the following we are going to build suitable sequences $(\delta_k)$ and $(n_k)$. \\

First of all, let us recall the following useful lemma to estimate the lower $A$-density of a given sequence $(n_k)$ 
\cite[Lemma 2.7]{ErMo}.

\begin{lemma}\label{LemmaDensInfCalc} Let $(\alpha_k)$ be a non-negative sequence such 
that $\sum_{k\in\N} \alpha_k=+\infty.$ Assume that the sequence $(\alpha_n/\sum_{j=1}^n\alpha_j)$ converges to 
zero as $n$ tends to $+\infty$. Let $(n_k)$ be an increasing sequence of integers. 
Then, we have 
$$\underline{d}_{A}((n_k)_k)=\liminf_{k\rightarrow +\infty}\left(\frac{\sum_{j=1}^{k}\alpha_{n_j}}{\sum_{j=1}^{n_k}\alpha_j}\right),$$
where $\underline{d}_{A}$ is the $A$-density given by the admissible matrix $A=(\alpha_k/\sum_{j=1}^n\alpha_j).$ 
\end{lemma}

For every positive integer $k,$ we define $\delta_k$ and $L_0$ as follows: 
$\delta_k=l$ and $L_0=l-1$ where $l$ is the place of the first zero in the 
dyadic representation of $k.$ For example consider $k=11,$ i.e. 
$k=1.2^0+1.2^1+0.2^2+1.2^3,$ we have $\delta_k=3$ and 
$L_0=2.$ Then we construct the following increasing sequence $(n_k)$ of positive integers by setting
\begin{equation}\label{ind_def}
n_1=2\hbox{ and }n_k=2\sum_{i=1}^{k-1}\delta_i+\delta_k\quad (k\geq 2)
\end{equation}

\begin{lemma}\label{puiss2} With the above notations, we have for every $m\geq 1,$
\[
n_{2^m}=4.2^m-3 \hbox{ and }n_{2^m-1}=4.2^m-m-5.
\]
\end{lemma}
\begin{proof} We have $n_1=n_{2-1}=2=4.2-1-5$ and $n_2=5=4.2-3.$ We set, for $k\geq 2,$  
\[
n_k'=\sum_{i=1}^k\delta_i=\frac{n_k+\delta_k}{2}.
\]
Let us consider the sets $\Delta_{j}^{(m)}=\{1\leq k\leq 2^m-1;\delta_k=j\},$ for $m\geq 1$ and $j=1,\dots,m+1.$ We get, for $m\geq 2,$  
\[
n_{2^m-1}'=\sum_{i=1}^{2^m-1}\delta_i=\sum_{j=1}^{m+1}j\#\Delta_{j}^{(m)}
\]
Thus it suffices to compute the cardinal of the sets $\Delta_{j}^{(m)}.$ For $j=m+1,$ there is only one possibility to 
obtain $\delta_l=m+1$ given by $l=(1,\dots,1,0,\dots)$ with a number $m$ of ones, i.e. $l=2^m-1.$ 
For $2\leq j\leq m,$ to obtain $\delta_l=j,$ $l$ has necessarily the following dyadic representation
\[
(\underbrace{1,1,\ldots,1}_{\text{length } j-1},0,\underbrace{\star,\star,\ldots,\star,\star}_{\text{length }m-j},0,0,\ldots).
\]
Thus we have $2^{m-j}$ possible choices. Finally to obtain $\delta_l=1,$ $l$ has necessarily the following dyadic representation
\[
(0,\underbrace{\star,\star,\ldots,\star,\star}_{\text{length }m-1},0,0,\ldots).
\]
and taking into account $l\ne 0,$ we have $2^{m-1}-1$ possible choices. Therefore we get, for all $m\geq 2,$
\[
n_{2^m-1}'=(2^{m-1}-1)+\sum_{j=2}^{m}j2^{m-j}+(m+1).
\]
An easy calculation gives
\[
n_{2^m-1}'=4.2^{m-1}-2\hbox{ for }m\geq 2.
\]
Therefore we deduce, for all $m\geq 1,$ 
\[
n_{2^m-1}=4.2^{m}-5-m.
\]
Finally we obtain, for every integer $m\geq 2,$ 
\[
n_{2^m}=n_{2^m-1}+\delta_{2^m-1}+\delta_{2^m}=n_{2^m-1}+(m+1)+1, 
\]
which leads to 
\[
n_{2^m}=4.2^m-3 \hbox{ for }m\geq 1.
\]
\end{proof}

\begin{lemma}\label{decompo} With the above notations, we have, for every positive integer $k$ which has the following dyadic representation 
$k:=2^n+\sum_{i=0}^{n-1}\alpha_i 2^i,$ $2^n<k< 2^{n+1}-1,$ $\alpha_i\in\{0,1\},$ 
\[
n_{k-2^n}=2\sum_{i=1+2^n}^{k-1}\delta_i +\delta_k
\]
\end{lemma}
\begin{proof} It suffices to observe that every integer $i$ with $1+2^n\leq i\leq k<2^{n+1}-1$ has a dyadic representation as
\[
(\underbrace{\star,\ldots,\star}_{\text{length }n},1,0\ldots)
\]
with at least a zero between the first and the $n^{th}$ position (since $k\ne 2^{n+1}-1$). Therefore we have 
$\delta_i= l$ for some $1\leq l\leq n$ and 
\[
i-2^n=(\underbrace{1,\ldots,1}_{\text{length }l-1},0,\underbrace{\star,\ldots,\star}_{\text{length }n-l},0,0\ldots)
\]
We deduce for $1+2^n\leq i\leq k,$ $\delta_{i-2^n}=\delta_i.$ Therefore we get, for every positive integer $k$ with the dyadic representation 
$k:=2^n+\sum_{i=0}^{n-1}\alpha_i 2^i,$ $2^n<k< 2^{n+1}-1,$ 
\[
n_{k-2^n}=2\sum_{i=1}^{k-1-2^n}\delta_i +\delta_{k-2^n}=2\sum_{i=1+2^n}^{k-1}\delta_i +\delta_{k}.
\] 
\end{proof}

\begin{lemma}\label{puiss3} With the above notations, we have, for every positive integer $k$ which has the following dyadic representation
$k:=\sum_{i=1+L_0}^{n}\alpha_i 2^i+2^{L_0}-1$ with $n\geq 1,$ $\alpha_n=1$ and $\alpha_i\in\{0,1\},$ $i=1+L_0,\dots,n,$  
\[
n_k=\left\{\begin{array}{l}\sum_{i=1+L_0}^{n}\alpha_i n_{2^i} +\sum_{i=1+L_0}^{n}\alpha_i +n_{2^{L_0}-1},\hbox{ if }1\leq L_0\leq n-1,\\
\sum_{i=1}^{n}\alpha_i n_{2^i} +\sum_{i=1}^{n}\alpha_i-1,\hbox{ if }L_0=0.
\end{array}\right.
\]
\end{lemma}
\begin{proof} \begin{itemize}
\item Case $L_0=0$: we have 
$k=\sum_{i=1}^{n-1}\alpha_i2^i+2^n,$ with $\alpha_n=1$ and $\alpha_i\in\{0,1\}.$ For $k=2^n,$ we have $\alpha_i=0,$ 
$i=1,\dots, n-1$ and we can write 
$n_k=n_{2^n}+\sum_{i=1}^{n}\alpha_i-1$ and the announced result holds. Otherwise, we define the sequence 
$1\leq l_1<l_2<\dots<l_m\leq n-1,$ satisfying $\alpha_{l_j}=1,$ for $j=1,\ldots,m$ and $\alpha_i=0$ for 
$i\notin\{l_1,\ldots,l_m\}.$ Taking into account Lemma \ref{decompo} we write
\begin{equation}\label{calcul1}
\begin{array}{rcl}n_k&=&\left(2\sum_{i=1}^{2^n-1}\delta_i+\delta_{2^n}\right)+\delta_{2^n}+\left(2\sum_{i=1+2^n}^{k-1}\delta_i+\delta_k\right)\\
&=&n_{2^n}+1+n_{k-2^n}\\
&=&n_{2^n}+\alpha_n+n_{k-2^n}.
\end{array}
\end{equation}
We get $k-2^n=\sum_{j=1}^{m}2^{l_j}$ and a calculation similar to 
(\ref{calcul1}) leads to 
\[
n_k=n_{2^n}+\alpha_n+n_{2^{l_m}}+\alpha_{l_m}+n_{k-2^n-2^{l_m}}=
\alpha_n n_{2^n}+\alpha_{l_m} n_{2^{l_m}}+\alpha_n+\alpha_{l_m}+n_{k-2^n-2^{l_m}}.
\]
By repeating the reasoning we obtain 
\[
n_k=\sum_{j=2}^{m} n_{2^{l_j}}+n_{2^n}+\sum_{j=2}^{m}\alpha_{l_j}+\alpha_n +n_{2^{l_1}}.
\]
Finally since we have $n_{2^l_1}=n_{2^l_1}+\alpha_{l_1}-1,$ we can write
\[
n_k=\sum_{i=1}^{n}\alpha_i n_{2^i} +\sum_{i=1}^{n}\alpha_i -1.
\]
\item Case $1\leq L_0\leq n-1$: first, if we have $\alpha_i=0,$ for $i=1+L_0,\dots,n-1,$ we get $k-2^n=2^{L_0}-1$ and $n_k=n_{2^{n}}+\alpha_n+n_{2^{L_0}-1}$ which had to be proved.
Otherwise, we set $m=\max(1+L_0\leq i\leq n-1;\alpha_i=1).$ We have $k-2^n=2^m+\sum_{i=1+L_0}^{m-1}\alpha_i 2^i+2^{L_0}-1$ and a calculation similar to 
(\ref{calcul1}) leads to 
\[
n_k=n_{2^n}+n_{2^m}+\alpha_n+\alpha_m+n_{k-2^n-2^m}=\alpha_n n_{2^n}+\alpha_m n_{2^m}+\alpha_n+\alpha_m+n_{k-2^n-2^m}.
\]
By repeating the reasoning we obtain 
\[
n_k=\sum_{i=1+L_0}^{n}\alpha_i n_{2^i} +\sum_{i=1+L_0}^{n}\alpha_i +n_{2^{L_0}-1}.
\]
\end{itemize}
\end{proof}

\begin{lemma}\label{estimation1eresuite} With the above notations, for every positive integer $k$ which has the following dyadic representation
$k:=\sum_{i=1+L_0}^{n}\alpha_i 2^i+2^{L_0}-1$ with $n\geq 1,$ $\alpha_n=1,$ $0\leq L_0\leq n-1$ and 
$\alpha_i\in\{0,1\},$ $i=1+L_0,\dots,n$ we have 
\[
n_k=4k-2\sum_{i=1+L_0}^{n}\alpha_i-L_0-1.
\]
On the other hand, for every positive integer $k$ with $L_0=n+1,$ we have $n_k=4k-L_0-1$ again.
\end{lemma}
\begin{proof} First we deal with the case $L_0=n+1,$ i.e. $k=2^{n+1}-1.$ Lemma \ref{puiss2} ensures that 
$n_{2^{n+1}-1}=4.2^{n+1}-(n+1)-5=4.(2^{n+1}-1)-(n+1)-1$ and we have the desired conclusion. 
If $L_0\ne n+1,$ we necessarily have $0\leq L_0\leq n-1.$ Let us consider the case $L_0=0:$ we apply Lemma \ref{puiss3} to write
\[
n_k=\sum_{i=1}^{n}\alpha_i n_{2^i} +\sum_{i=1}^{n}\alpha_i-1.
\]
Using Lemma \ref{puiss2} we deduce
\[\begin{array}{rcl}n_k&=&\sum_{i=1}^n\alpha_i(4.2^i-3)+\sum_{i=1}^n\alpha_i-1\\
&=&4\sum_{i=1}^n\alpha_i2^i-2\sum_{i=1}^n\alpha_i-1\\
&=&4k-2\sum_{i=1}^{n}\alpha_i-1=4k-2\sum_{i=1}^{n}\alpha_i-L_0-1.
\end{array}
\]
Otherwise, we consider the case $1\leq L_0\leq n-1$ and we apply Lemma \ref{puiss3} again to write
\[
n_k=\sum_{i=1+L_0}^{n}\alpha_i n_{2^i} +\sum_{i=1+L_0}^{n}\alpha_i +n_{2^{L_0}-1}.
\]
We conclude by using Lemma \ref{puiss2},
\[\begin{array}{rcl}n_k&=&\sum_{i=1+L_0}^n\alpha_i(4.2^i-3)+\sum_{i=1+L_0}^n\alpha_i+4.2^{L_0}-L_0-5\\
&=&4\left(\sum_{i=1+L_0}^n\alpha_i2^i+2^{L_0}-1\right)-1-2\sum_{i=1+L_0}^n\alpha_i-L_0\\
&=&4k-2\sum_{i=1+L_0}^{n}\alpha_i-L_0-1.
\end{array}
\]
\end{proof}

\begin{proposition}\label{estim1eresuitetotal} The sequence $(n_k)$ satisfies the following optimal estimate: for every integer $k\geq 2,$
\[
4k-2\lfloor\log_2(k)\rfloor-1\leq n_k\leq 4k-3
\]
\end{proposition}

\begin{proof} We begin with the case $k=2^n-1,$ $n\geq2.$ We have $n_k=4.2^n-n-5=4k-2-\lfloor\log_2(k)\rfloor\geq 4k-2\lfloor\log_2(k)\rfloor-1.$ Now let us 
consider $k\geq 2,$ with $k\neq 2^n-1$ i.e. $\log_2(k+1)\notin\mathbb{N}.$ The integer $k$ necessarily has the following representation 
$k=2^n+\sum_{i=1+L_0}^{n}\alpha_i 2^i+2^{L_0}-1$ with $n\geq 1,$ and $\alpha_i\in\{0,1\},$ $i=1+L_0,\dots,n-1.$ 
Lemma \ref{estimation1eresuite} gives $n_k=4k-2\sum_{i=1+L_0}^{n}\alpha_i-L_0-1.$ 
Clearly we have $n_k\leq 4k-3$ and $n_{2^n}=4.2^n-3$ by Lemma \ref{puiss2}. On the other hand we have
\[
n_k\geq 4k-2(n-L_0)-L_0-1\geq 4k-2n-1=4k-2\lfloor\log_2(k)\rfloor-1.
\]  
Finally for $k=2^{m+1}-2,$ we have 
$k=\sum_{i=1}^{m}2^i,$ $L_0=0,$ $\lfloor\log_2(2^{m+1}-2)\rfloor=m$ and Lemma \ref{puiss2} gives $n_{2^{m+1}-2}=4(2^{m+1}-2)-2m-1=4(2^{m+1}-2)-2\lfloor\log_2(2^{m+1}-2)\rfloor-1=4k-2\lfloor\log_2(k)\rfloor-1$.
\end{proof}

\begin{proposition} The sequence $(n_k)$ defined above satisfies $\underline{d}_{B_1}((n_k)_k)>0.$
\end{proposition}

\begin{proof} Using Lemma \ref{LemmaDensInfCalc} the following equality holds
\[
\underline{d}_{B_1}((n_k)_k)=\liminf_{k\rightarrow +\infty} \left(\frac{\sum_{j=2}^{k}e^{n_j/\log(n_j)}}{\sum_{j=2}^{n_k}e^{j/\log(j)}}\right).
\]
By a classical calculation, we obtain $\sum_{j=2}^{n_k}e^{j/\log(j)}\sim \log (n_k)e^{n_k/\log(n_k)}$ as $k\rightarrow +\infty.$ Moreover 
according to Proposition \ref{estim1eresuitetotal}, there exists a constant $C>0$ such that, for $N$ large enough and $k\geq N$,
\[
\frac{\sum_{j=N}^{k}e^{n_j/\log(n_j)}}{\log (n_k)e^{n_k/\log(n_k)}}\geq \frac{\sum_{j=N}^{k}e^{(4j-C\log (j))/\log(4j-C\log (j))}}{\log (4k)e^{4k/\log(4k)}}.
\]
With a summation by parts, we get 
\[
\sum_{j=N}^{k}e^{(4j-C\log (j))/\log(4j-C\log (j))}\sim \frac{\log (k)}{4}e^{(4k-C\log (k))/\log(4k-C\log (k))}\hbox{ as }k\rightarrow +\infty.
\]
Finally a simple computation leads to 
\[
\frac{\log (k)}{4}\frac{e^{(4k-C\log (k))/\log(4k-C\log (k))}}{\log (4k)e^{4k/\log(4k)}}\rightarrow \frac{e^{-C}}{4}\hbox{ as }k\rightarrow +\infty.
\]
This allows to conclude $\underline{d}_{B_1}((n_k)_k)\geq\frac{e^{-C}}{4}>0.$
\end{proof}

Hence we deduce the following result, which improves \cite[Theorem 4.12]{ErMo}.

\begin{proposition} Let $X$ be a Fr\'echet space, $Y$ a separable Fr\'echet space and 
$T_n:X\rightarrow Y,$ $n\in\mathbb{N},$ continuous mappings. If 
the sequence $(T_n)$ satisfies the frequent universality criterion, then $(T_n)$ is $B_1$-frequently universal. 
\end{proposition}

\section{Further results}\label{cons_spec_sequ2} We are going to modify the sequence $(n_k)$ built in the preceding section to obtain a new sequence with positive $A$-density for an admissible matrix $A$ 
defining a sharper density than the natural density and the $B_1$-density. This construction is inspired by Section 4 of \cite{ErMo}. Let us consider an increasing sequence $(a_n)$ 
of positive integers with $a_1=1$. Then we define the function $f:\mathbb{N}\rightarrow\mathbb{N},$ by 
$f(j)=m$ for all $j\in\{a_m,\ldots,a_{m+1}-1\}.$ We also define the sequence $(n_k(f))$ by induction as in (\ref{ind_def}): 
$$n_1(f)=2\hbox{ and }n_{k}(f)=n_{k-1}(f)+f(\delta_{k-1})+f(\delta_{k})\hbox{ for }k\geq 2.$$ Clearly we obtain the following equality, for all $k\geq 2,$
\begin{equation}\label{Eqnkmod}
n_{k}(f)=2\sum_{i=1}^{k-1}f(\delta_{i})+f(\delta_{k}).
\end{equation}
Observe that, for all $k\ne l,$ 
\begin{equation}\label{new_form}
\vert n_{k}(f)-n_{l}(f)\vert\geq f(\delta_{k})+f(\delta_l).
\end{equation}
In the previous section, the sequence $(n_k)$ satisfied $\vert n_k-n_l\vert\geq \delta_k+\delta_l$ in order to prove the frequent universality criterion. In a previous work \cite{ErMo}, we already proved that this condition could be relaxed as in 
(\ref{new_form}), provided that $f$ increases and tends to infinity. 
Let us notice that, if we set $a_m=m$ for every $m\geq 1,$ then the corresponding sequence 
$(n_k(f))$ is the sequence $(n_k)$ of Section \ref{cons_spec_sequ}. Throughout this section, we will omit the notation $f$ in 
$(n_k(f))$ for sake of readability. Thus our purpose is to compute an exact formula for the new sequence $(n_k)$ to understand its asymptotic behavior and 
to obtain sharper estimates for the densities given by the frequent universality criterion.\\ 

First of all, 
we are going to obtain an expression for $n_{2^{a_m+q}},$ with $0\leq q<a_{m+1}-a_m.$ 

\begin{lemma}\label{lem1}
For every $m\in\mathbb{N}$ and every $0\leq q<a_{m+1}-a_m$, we have
\[n_{2^{a_m+q}}=-1-2m+2f(1+a_m+q)+2^{a_m+q+1}\sum_{i=1}^{m}\frac{1}{2^{a_{i}-1}}.\]
\end{lemma}

\begin{proof}
Let us define $\Delta_j^{(m,q)}:=\{1\leq l\leq 2^{a_m+q}-1: \delta_l=j\}$ for $j\geq1$. It is clear that for every $1\leq l\leq 2^{a_m+q}-2$, the first zero in the dyadic decomposition of $l$ appears in position less than $a_m+q$ i.e. $\delta_l\leq a_m+q$ and $\delta_{2^{a_m+q}-1}=a_m+q+1$. Thus, for every $j> a_m+q+1$ we have $\#\Delta_j^{(m,q)}=0$ and $\#\Delta_{a_m+q+1}^{(m,q)}=1$.
Let us now compute $\#\Delta_j^{(m,q)}$ for $1\leq j\leq a_m+q$. 
First, let us observe that we have $j=\delta_l=1$ if and only if $l$ is even. From this, we deduce that $\#\Delta_1^{(m,q)}=2^{a_m+q-1}-1$. 
On the other hand, if $1\leq l\leq 2^{a_m+q}-2$ is such that $j=\delta_l\geq2$ then its dyadic decomposition is given by a $j-1$ ones followed by one zero and then we have $2^{a_m+q-j}$ choices as shown in the figure below:
\[
(\overbrace{\underbrace{1,1,\ldots,1,0}_{\text{length } j},\star,\star,\ldots,\star,\star}^{\text{length } a_m+q},0,0,\ldots).
\]
Thus we obtain that for every $2\leq j\leq a_m+q$, $\#\Delta_j^{(m,q)}=2^{a_m+q-j}$. 
We use these facts to compute $n_{2^{a_m+q}},$ for $q$ satisfying $a_m\leq a_m+q<a_{m+1}.$
\begin{align*}
n_{2^{a_m+q}}&=2\sum_{j=1}^{2^{a_m+q}-1}f(\delta_j)+f(\delta_{2^{a_m+q}})&\\
&=2\sum_{j=1}^{a_m+q+1}f(j)\#\Delta_j^{(m,q)}+f(\delta_{2^{a_m+q}})&\\
&=2\left(\left(2^{a_m+q-1}-1\right)f(1)+\sum_{j=2}^{a_m+q}f(j)2^{a_m+q-j}+f(1+a_m+q)\right)+1\text{, using }f(\delta_{2^{a_m+q}})=f(1)=1&\\
&=2\left(-1 +\sum_{i=1}^{m-1}\sum_{j=a_i}^{a_{i+1}-1}f(j)2^{a_m+q-j}+\sum_{j=a_m}^{a_m+q}f(j)2^{a_m+q-j}+f(1+a_m+q)\right)+1.&
\end{align*}
Setting $u_m=\sum_{i=1}^{m-1}\sum_{j=a_i}^{a_{i+1}-1}f(j)2^{a_m+q-j}+\sum_{j=a_m}^{a_m+q}f(j)2^{a_m+q-j},$ we have

\[\begin{array}{rcl}
u_m&=&
\displaystyle\sum_{i=1}^{m-1}i2^{a_m+q}\sum_{j=a_i}^{a_{i+1}-1}2^{-j}+m\sum_{j=a_m}^{a_m+q}2^{a_m+q-j}\\&=&
\displaystyle 2^{a_m+q}\left(\sum_{i=1}^{m-1}\left(\frac{i}{2^{a_i-1}}-\frac{i}{2^{a_{i+1}-1}}\right)+m\left(\frac{1}{2^{a_m-1}}-\frac{1}{2^{a_m+q}}\right)\right)
\\&=&
\displaystyle 2^{a_m+q}\left( \sum_{i=1}^{m}\frac{1}{2^{a_{i}-1}}-\frac{m}{2^{a_m+q}}\right).
\end{array}
\]

Therefore we get

\[
n_{2^{a_m+q}}=-1-2m+2f(1+a_m+q)+2^{a_m+q+1}\sum_{i=1}^{m}\frac{1}{2^{a_{i}-1}}.
\]

\end{proof}

From this lemma we get the following result. 

\begin{lemma}\label{lem2} For every $m\in\mathbb{N}$ and every $0\leq q <a_{m+1}-a_m$, we have
	\[n_{2^{a_m+q}}=\begin{cases}
	-1+2^{a_m+q+1}\sum_{i=1}^{m}\frac{1}{2^{a_i-1}}&\text{ if }0\leq q<a_{m+1}-a_m-1\\
	1+2^{a_{m+1}}\sum_{i=1}^{m}\frac{1}{2^{a_i-1}}&\text{ if }q=a_{m+1}-a_m-1
	\end{cases}\]
\end{lemma}

\begin{proof}
It suffices to apply Lemma \ref{lem1} taking into account that we have $f(a_m+q+1)=m$ if $0\leq q\leq a_{m+1}-a_m-1$ and $f(a_m+q+1)=m+1$ if $q=a_{m+1}-a_m-1$.
\end{proof}

In the sequel, we will need to have an expression for the integers $n_{2^L-1}$ for $L\geq 1.$ This is the goal of the following lemma. 

\begin{lemma}\label{2bis} Let $L\geq 2$ with $a_{l}-1\leq L<a_{l+1}-1.$ Then we have
$$n_{2^L-1}=\begin{cases}
2^{a_l}\sum_{i=1}^{l-1}\frac{1}{2^{a_i-1}}-l&\text{ for } L=a_{l}-1\\
2^{1+L}\sum_{i=1}^{l}\frac{1}{2^{a_i-1}}-(l+2)&\text{ for } L>a_{l}-1
\end{cases}$$
\end{lemma}

\begin{proof} By definition the following equality holds
\begin{equation}\label{reaut}
n_{2^L}=n_{2^L-1}+f(\delta_{2^L-1})+f(\delta_{2^L})=n_{2^L-1}+f(1+L)+1.
\end{equation}
First, let us consider the case $L=a_l-1=a_{l-1}+(a_l-a_{l-1}-1).$ Lemma \ref{lem2} gives $n_{2^L}=n_{2^{a_l-1}}=1+2^{a_l}\sum_{i=1}^{l-1}\frac{1}{2^{a_i-1}}$.
Moreover, since we have $f(1+L)=f(a_l)=l,$ then formula (\ref{reaut}) implies 
\[
n_{2^{a_l-1}-1}=n_{2^{a_l-1}}-l-1=2^{a_l}\sum_{i=1}^{l-1}\frac{1}{2^{a_i-1}}-l.
\]
On the other hand,  if we have $a_l-1<L<a_{l+1}-1$, Lemma \ref{lem2} yields:
$n_{2^L}=-1+2^{L+1}\sum_{i=1}^{l}\frac{1}{2^{a_i-1}}$.
Thus, formula (\ref{reaut}) gives:
\[
n_{2^L-1}=2^{1+L}\sum_{i=1}^{l}\frac{1}{2^{a_i-1}}-(l+2).
\]
\end{proof}

\begin{lemma}\label{lem3}
Let $N=2^n+\sum_{i=0}^{n-1}\alpha_i 2^i$ with $n\geq 1$ and $\sum_{i=0}^{n-1}\alpha_i 2^i\neq0$. Then the following holds
\[
n_{N-2^n}=\begin{cases}
2\sum_{i=2^n+1}^{N-1} f(\delta_i)+f(\delta_N)&\text{ if } N<2^{n+1}-1\\
2\sum_{i=2^n+1}^{N-1} f(\delta_i)+2f(\delta_N)-f(\delta_{N-2^n})&\text{ if } N=2^{n+1}-1\\
\end{cases}
\]
\end{lemma}

\begin{proof} We easily adapt the proof of Lemma \ref{decompo}.
\end{proof}

Using Lemma \ref{lem3} instead of Lemma \ref{decompo}, we may also prove that Lemma \ref{puiss3} remains valid in this context. We shall use it in the sequel.

\begin{definition}\label{defdecompoN} 
{\rm In what follows, we express every positive integer $N$ in the following fashion, taking into account the properties of the sequence 
$(a_l)$,
\[
N=2^{L_0}-1+\sum_{q=q_0}^{a_{l_0}-a_{l_0-1}-1}\alpha_{a_{l_0-1}+q}2^{a_{l_0-1}+q}+\sum_{j=l_0}^{l_N-1}\sum_{q=0}^{a_{j+1}-a_j-1}\alpha_{a_{j}+q}2^{a_j+q}+\sum_{q=0}^{q_N}\alpha_{a_{l_N}+q}2^{a_{l_N}+q}
\]
with $a_{l_0-1}\leq 1+L_0<a_{l_0}$, $a_{l_N}\leq q_N+a_{l_N}<a_{l_{N}+1}$, $q_0+a_{l_0-1}=1+L_0$ and $\alpha_{a_{l_N}+q_N}=1$. We also set $w_N=q_N+a_{l_N}$.}
\end{definition}

\begin{lemma}\label{lem5}
For every positive integer $N,$ we have, using the notations of Definition \ref{defdecompoN}, 
\begin{align*}
n_N&=n_{2^{L_0}-1}+2\left(\sum_{j=l_0}^{1+l_N}\alpha_{a_{j}-1}\right)+\sum_{q=q_0}^{a_{l_0}-1-a_{l_0-1}}\alpha_{a_{l_0-1}+q}2^{a_{l_0-1}+q+1}\left(\sum_{i=1}^{l_0-1}\frac{1}{2^{a_i-1}}\right)\\
&+\sum_{j=l_0}^{l_N-1}\sum_{q=0}^{a_{j+1}-a_j-1}\alpha_{a_{j}+q}2^{a_j+q+1}\left(\sum_{i=1}^{j}\frac{1}{2^{a_i-1}}\right)+\sum_{q=0}^{q_N}\alpha_{a_{l_N}+q}2^{a_{l_N}+q+1}\left(\sum_{i=1}^{l_N}\frac{1}{2^{a_i-1}}\right).
\end{align*}
\end{lemma}

\begin{proof} Using the notations of Definition \ref{defdecompoN}, Lemma \ref{puiss3} gives the following equality (with the convention $n_0=-1$)
Since $N=\sum_{i=L_0+1}^{n}\alpha_i2^i+\left(2^{L_0}-1\right)=\sum_{i=0}^{n}\alpha_i2^i$, 
\begin{equation}\label{split} 
\begin{array}{rl}
n_N=&\displaystyle n_{2^{L_0}-1}+\sum_{q=q_0}^{a_{l_0}-1-a_{l_0-1}}
\alpha_{a_{l_0-1}+q}n_{2^{a_{l_0-1}+q}}+\sum_{j=l_0}^{l_N-1}\sum_{q=0}^{a_{j+1}-a_j-1}\alpha_{a_{j}+q}n_{2^{a_j+q}}\\
&+\displaystyle\sum_{q=0}^{q_N}\alpha_{a_{l_N}+q}n_{2^{a_{l_N}+q}}+\sum_{i=1+L_0}^{w_N}\alpha_i.
\end{array}
\end{equation}

We begin by transforming the first sum above. We apply Lemma \ref{lem2} and we get
\begin{align*}
&\sum_{q=q_0}^{a_{l_0}-1-a_{l_0-1}}\alpha_{a_{l_0-1}+q}n_{2^{a_{l_0-1}+q}}\\
&=\sum_{q=q_0}^{a_{l_0}-2-a_{l_0-1}}\alpha_{a_{l_0-1}+q}\left(-1+2^{a_{l_0-1}+q+1}\left(\sum_{i=1}^{l_0-1}\frac{1}{2^{a_i-1}}\right)\right)+\alpha_{a_{l_0}-1}\left(1+2^{a_{l_0}}\sum_{i=1}^{l_0-1}\frac{1}{2^{a_i-1}}\right)\\
&=\sum_{q=q_0}^{a_{l_0}-1-a_{l_0-1}}\alpha_{a_{l_0-1}+q}2^{a_{l_0-1}+q+1}\left(\sum_{i=1}^{l_0-1}\frac{1}{2^{a_i-1}}\right)-\sum_{q=q_0}^{a_{l_0}-2-a_{l_0-1}}\alpha_{a_{l_0-1}+q}+\alpha_{a_{l_0}-1}.
\end{align*}

We proceed in the same way with the second sum of (\ref{split}) and we obtain

\begin{align*}
&\sum_{j=l_0}^{l_N-1}\sum_{q=0}^{a_{j+1}-a_j-1}\alpha_{a_{j}+q}n_{2^{a_j+q}}\\
&=\sum_{j=l_0}^{l_N-1}\sum_{q=0}^{a_{j+1}-a_j-2}\alpha_{a_{j}+q}\left(-1+2^{a_{j}+q+1}\left(\sum_{i=1}^{j}\frac{1}{2^{a_i-1}}\right)\right)+\sum_{j=l_0}^{l_N-1}\alpha_{a_{j+1}-1}\left(1+2^{a_{j+1}}\left(\sum_{i=1}^{j}\frac{1}{2^{a_i-1}}\right)\right)\\
&=\sum_{j=l_0}^{l_N-1}\sum_{q=0}^{a_{j+1}-a_j-1}\alpha_{a_{j}+q}2^{a_{j}+q+1}\left(\sum_{i=1}^{j}\frac{1}{2^{a_i-1}}\right)-\sum_{j=l_0}^{l_N-1}\sum_{q=0}^{a_{j+1}-a_j-2}\alpha_{a_{j}+q}+\sum_{j=l_0}^{l_N-1}\alpha_{a_{j+1}-1}.
\end{align*}

Finally to study the third sum of (\ref{split}), we consider two cases. 
\begin{itemize}
\item Case $w_N<a_{l_{N}+1}-1$: 

\[
\begin{array}{rcl}\displaystyle
\sum_{q=0}^{q_N}\alpha_{a_{l_N}+q}n_{2^{a_{l_N}+q}}&=&\displaystyle
\sum_{q=0}^{q_N}\alpha_{a_{l_N}+q}\left(-1+2^{a_{l_N}+q+1}\left(\sum_{i=1}^{l_N}\frac{1}{2^{a_i-1}}\right)\right)\\
&=&\displaystyle\sum_{q=0}^{q_N}\alpha_{a_{l_N}+q}2^{a_{l_N}+q+1}\left(\sum_{i=1}^{l_N}\frac{1}{2^{a_i-1}}\right)-\sum_{q=0}^{q_N}\alpha_{a_{l_N}+q}.
\end{array}
\]

\item Case $w_N=a_{l_{N}+1}-1$: 

\begin{align*}
&\sum_{q=0}^{q_N}\alpha_{a_{l_N}+q}n_{2^{a_{l_N}+q}}\\
&=\sum_{q=0}^{a_{l_N+1}-2-a_{l_N}}\alpha_{a_{l_N}+q}\left(-1+2^{a_{l_N}+q+1}\left(\sum_{i=1}^{l_N}\frac{1}{2^{a_i-1}}\right)\right)+\alpha_{a_{l_N+1}-1}\left(1+2^{a_{l_N+1}}\left(\sum_{i=1}^{l_N}\frac{1}{2^{a_i-1}}\right)\right)\\
&=\sum_{q=0}^{a_{l_N+1}-1-a_{l_N}}\alpha_{a_{l_N}+q}2^{a_{l_N}+q+1}\left(\sum_{i=1}^{l_N}\frac{1}{2^{a_i-1}}\right)-\sum_{q=0}^{a_{l_N+1}-2-a_{l_N}}\alpha_{a_{l_N}+q}+\alpha_{a_{l_N+1}-1}.
\end{align*}

\end{itemize}

Now, we gather these results in the case where $w_N<a_{l_{N}+1}-1$ (the case $w_N=a_{l_{N}+1}-1$ being similar). We get 

\begin{align*}
n_N&=n_{2^{L_0}-1}+\sum_{q=q_0}^{a_{l_0}-1-a_{l_0-1}}\alpha_{a_{l_0-1}+q}n_{2^{a_{l_0-1}+q}}+\sum_{j=l_0}^{l_N-1}\sum_{q=0}^{a_{j+1}-a_j-1}\alpha_{a_{j}+q}n_{2^{a_j+q}}\\
&+\sum_{q=0}^{q_N}\alpha_{a_{l_N}+q}n_{2^{a_{l_N}+q}}+\sum_{i=1+L_0}^{w_N}\alpha_i\\
&=n_{2^{L_0}-1}+\sum_{q=q_0}^{a_{l_0}-1-a_{l_0-1}}\alpha_{a_{l_0-1}+q}2^{a_{l_0-1}+q+1}\left(\sum_{i=1}^{l_0-1}\frac{1}{2^{a_i-1}}\right)\\
&+\sum_{j=l_0}^{l_N-1}\sum_{q=0}^{a_{j+1}-a_j-1}\alpha_{a_{j}+q}2^{a_{j}+q+1}\left(\sum_{i=1}^{j}\frac{1}{2^{a_i-1}}\right)+\sum_{q=0}^{q_N}\alpha_{a_{l_N}+q}2^{a_{l_N}+q+1}\left(\sum_{i=1}^{l_N}\frac{1}{2^{a_i-1}}\right)\\
&-\sum_{q=q_0}^{a_{l_0}-2-a_{l_0-1}}\alpha_{a_{l_0-1}+q}-\sum_{j=l_0}^{l_N-1}\sum_{q=0}^{a_{j+1}-a_j-2}\alpha_{a_{j}+q}-\sum_{q=0}^{q_N}\alpha_{a_{l_N}+q}+\alpha_{a_{l_0}-1}+\sum_{j=l_0}^{l_N-1}\alpha_{a_{j+1}-1}+\sum_{i=1+L_0}^{w_N}\alpha_i.
\end{align*}

Moreover, observe that, by definition, if $w_N<a_{l_{N}+1}-1$,
\begin{align*}
\sum_{i=1+L_0}^{w_N}\alpha_i&=\sum_{q=q_0}^{a_{l_0}-1-a_{l_0-1}}\alpha_{a_{l_0-1}+q}+\sum_{j=l_0}^{l_N-1}\sum_{q=0}^{a_{j+1}-a_j-1}\alpha_{a_{j}+q}+\sum_{q=0}^{q_N}\alpha_{a_{l_N}+q}\\
&=\sum_{q=q_0}^{a_{l_0}-2-a_{l_0-1}}\alpha_{a_{l_0-1}+q}+\sum_{j=l_0}^{l_N-1}\sum_{q=0}^{a_{j+1}-a_j-2}\alpha_{a_{j}+q}+\sum_{q=0}^{q_N}\alpha_{a_{l_N}+q}+\alpha_{a_{l_0}-1}+\sum_{j=l_0}^{l_N-1}\alpha_{a_{j+1}-1}.
\end{align*}

Let us observe also that if $w_N<a_{l_{N}+1}-1$, then we have $\alpha_{a_{1+l_N}-1}=0$. Thus, we derive the announced result
\[
\begin{array}{rcl}
n_N&=&\displaystyle n_{2^{L_0}-1}+\sum_{q=q_0}^{a_{l_0}-1-a_{l_0-1}}\alpha_{a_{l_0-1}+q}2^{a_{l_0-1}+q+1}
\left(\sum_{i=1}^{l_0-1}\frac{1}{2^{a_i-1}}\right)\\&&\displaystyle 
+\sum_{j=l_0}^{l_N-1}\sum_{q=0}^{a_{j+1}-a_j-1}\alpha_{a_{j}+q}2^{a_{j}+q+1}\left(\sum_{i=1}^{j}\frac{1}{2^{a_i-1}}\right)\\
&&+\displaystyle\sum_{q=0}^{q_N}\alpha_{a_{l_N}+q}2^{a_{l_N}+q+1}\left(\sum_{i=1}^{l_N}\frac{1}{2^{a_i-1}}\right)+2\alpha_{a_{l_0}-1}+2\sum_{j=l_0}^{l_N-1}\alpha_{a_{j+1}-1}\\
&=& \displaystyle n_{2^{L_0}-1}+\sum_{q=q_0}^{a_{l_0}-1-a_{l_0-1}}\alpha_{a_{l_0-1}+q}2^{a_{l_0-1}+q+1}
\left(\sum_{i=1}^{l_0-1}\frac{1}{2^{a_i-1}}\right)\\&&\displaystyle +
\sum_{j=l_0}^{l_N-1}\sum_{q=0}^{a_{j+1}-a_j-1}\alpha_{a_{j}+q}2^{a_{j}+q+1}\left(\sum_{i=1}^{j}\frac{1}{2^{a_i-1}}\right)\\
&&+\displaystyle\sum_{q=0}^{q_N}\alpha_{a_{l_N}+q}2^{a_{l_N}+q+1}\left(\sum_{i=1}^{l_N}\frac{1}{2^{a_i-1}}\right)+2\sum_{j=l_0}^{1+l_N}\alpha_{a_{j}-1}.
\end{array}
\]
\end{proof}

\begin{lemma}\label{lem6} Using the notations of Definition \ref{defdecompoN}, we have, for every positive integer $N,$
\[
\begin{array}{rcl}
n_N=&\displaystyle 2N\left(\sum_{i=1}^{\infty}\frac{1}{2^{a_i-1}}\right)
+2\left(\sum_{i=1}^{\infty}\frac{1}{2^{a_i-1}}\right)+2\left(\sum_{j=l_0}^{1+l_N}\alpha_{a_{j}-1}\right)
-2^{1+L_0}\left(\sum_{i=l_0-\tau_{0}}^{\infty}\frac{1}{2^{a_i-1}}\right)\\
&-\displaystyle 2\left(\sum_{i=l_0}^{\infty}\frac{1}{2^{a_i-1}}\right)\left(\sum_{q=q_0}^{a_{l_0}-a_{l_0-1}-1}\alpha_{a_{l_0-1}+q}2^{a_{l_0-1}+q}\right)
\\&\displaystyle -2\sum_{j=l_0}^{l_N-1}\left(\sum_{i=j+1}^{\infty}\frac{1}{2^{a_i-1}}\right)\left(\sum_{q=0}^{a_{j+1}-a_j-1}\alpha_{a_j+q}2^{a_j+q}\right)\\
&\displaystyle -2\left(\sum_{i=l_N+1}^{\infty}\frac{1}{2^{a_i-1}}\right)\left(\sum_{q=0}^{q_N}\alpha_{l_N+q}2^{a_{l_N}+q}\right)-l_0-1+2\tau_0,
\end{array}
\]
with $\tau_0=0$ if $L_0>a_{l_0-1}-1$ and $\tau_0=1$ if $L_0=a_{l_0-1}-1$.
\end{lemma}

\begin{proof}
Let $N=\sum_{i=L_0+1}^{n}\alpha_i2^i+\left(2^{L_0}-1\right)=\sum_{i=0}^{n}\alpha_i2^i$ with $a_{l_0-1}-1\leq L_0 <a_{l_0}-1$.
Lemma \ref{lem5} gives:
\begin{align*}
n_N&=n_{2^{L_0}-1}+2\left(\sum_{j=l_0}^{1+l_N}\alpha_{a_{j}-1}\right)+\sum_{q=q_0}^{a_{l_0}-1-a_{l_0-1}}\alpha_{a_{l_0-1}+q}2^{a_{l_0-1}+q+1}\left(\sum_{i=1}^{l_0-1}\frac{1}{2^{a_i-1}}\right)\\
&+\sum_{j=l_0}^{l_N-1}\sum_{q=0}^{a_{j+1}-a_j-1}\alpha_{a_{j}+q}2^{a_j+q+1}\left(\sum_{i=1}^{j}\frac{1}{2^{a_i-1}}\right)+\sum_{q=0}^{q_N}\alpha_{a_{l_N}+q}2^{a_{l_N}+q+1}\left(\sum_{i=1}^{l_N}\frac{1}{2^{a_i-1}}\right).
\end{align*}
Let us now express every sum of the form $\sum_{i=1}^{K}\frac{1}{2^{a_i-1}}$ as $\sum_{i=1}^{\infty}\frac{1}{2^{a_i-1}}-\sum_{i=K+1}^{\infty}\frac{1}{2^{a_i-1}}$. This gives:
$$\begin{array}{ll}
n_N=&\ n_{2^{L_0}-1}+2\left(\sum_{j=l_0}^{1+l_N}\alpha_{a_{j}-1}\right)+
2\left(\sum_{i=1}^{\infty}\frac{1}{2^{a_i-1}}\right) \left(\sum_{q=q_0}^{a_{l_0}-a_{l_0-1}-1}\alpha_{a_{l_0-1}+q}2^{a_{l_0-1}+q}\right)
\\
&+2\left(\sum_{i=1}^{\infty}\frac{1}{2^{a_i-1}}\right) 
\left(\sum_{j=l_0}^{l_N-1}\sum_{q=0}^{a_{j+1}-a_j-1}\alpha_{a_{j}+q}2^{a_j+q}+\sum_{q=0}^{q_N}\alpha_{a_{l_N}+q}2^{a_{l_N}+q}\right)\\
&-2\left(\sum_{i=l_0}^{\infty}\frac{1}{2^{a_i-1}}\right)\left(\sum_{q=q_0}^{a_{l_0}-a_{l_0-1}-1}\alpha_{a_{l_0-1}+q}2^{a_{l_0-1}+q}\right)
\\&-2 \left(\sum_{j=l_0}^{l_N-1}\left(\sum_{i=j+1}^{\infty}\frac{1}{2^{a_i-1}}\right)\sum_{q=0}^{a_{j+1}-a_j-1}\alpha_{a_{j}+q}2^{a_j+q}\right)\\
&-2\left(\sum_{i=l_N+1}^{\infty}\frac{1}{2^{a_i-1}}\right)\left(\sum_{q=0}^{q_N}\alpha_{a_{l_N}+q}2^{a_{l_N}+q}\right)
\end{array}$$
Moreover, Lemma \ref{2bis} allows to express
\begin{align*}
n_{2^{L_0}-1}&=\begin{cases}
2^{1+L_0}\sum_{i=1}^{l_0-2}\frac{1}{2^{a_i-1}}-(l_0-1)&\text{ if } L_0=a_{l_0-1}-1\\
2^{1+L_0}\sum_{i=1}^{l_0-1}\frac{1}{2^{a_i-1}}-(l_0+1)&\text{ if } L_0>a_{l_0-1}-1
\end{cases}\\
&=\begin{cases}
2^{1+L_0}\left(\sum_{i=1}^{\infty}\frac{1}{2^{a_i-1}}\right)-2^{1+L_0}\left(\sum_{i=l_0-1}^{\infty}\frac{1}{2^{a_i-1}}\right)-(l_0-1)&\text{ if } L_0=a_{l_0-1}-1\\
2^{1+L_0}\left(\sum_{i=1}^{\infty}\frac{1}{2^{a_i-1}}\right)-2^{1+L_0}\left(\sum_{i=l_0}^{\infty}\frac{1}{2^{a_i-1}}\right)-(l_0+1)&\text{ if } L_0>a_{l_0-1}-1
\end{cases}
\end{align*}
From now on, we treat only the case when $L_0=a_{l_0-1}-1$, the other case being completely similar.
Let us replace in the first expression the value of $n_{2^{L_0}-1}$, we obtain:
$$\begin{array}{ll}
n_N=&\ -l_0+1-2^{1+L_0}\left(\sum_{i=l_0-1}^{\infty}\frac{1}{2^{a_i-1}}\right)\\&+2\left(\sum_{j=l_0}^{1+l_N}\alpha_{a_{j}-1}\right)
+2\left(\sum_{i=1}^{\infty}\frac{1}{2^{a_i-1}}\right) \left(2^{L_0}+\sum_{q=q_0}^{a_{l_0}-a_{l_0-1}-1}\alpha_{a_{l_0-1}+q}2^{a_{l_0-1}+q}\right)\\
&+2\left(\sum_{i=1}^{\infty}\frac{1}{2^{a_i-1}}\right) \left(\sum_{j=l_0}^{l_N-1}\sum_{q=0}^{a_{j+1}-a_j-1}\alpha_{a_{j}+q}2^{a_j+q}+\sum_{q=0}^{q_N}\alpha_{a_{l_N}+q}2^{a_{l_N}+q}\right)\\
&-2\left(\sum_{i=l_0}^{\infty}\frac{1}{2^{a_i-1}}\right)\left(\sum_{q=q_0}^{a_{l_0}-a_{l_0-1}-1}\alpha_{a_{l_0-1}+q}2^{a_{l_0-1}+q}\right)
\\&-2 \left(\sum_{j=l_0}^{l_N-1}\left(\sum_{i=j+1}^{\infty}\frac{1}{2^{a_i-1}}\right)\sum_{q=0}^{a_{j+1}-a_j-1}\alpha_{a_{j}+q}2^{a_j+q}\right)\\
&-2\left(\sum_{i=l_N+1}^{\infty}\frac{1}{2^{a_i-1}}\right)\left(\sum_{q=0}^{q_N}\alpha_{a_{l_N}+q}2^{a_{l_N}+q}\right)
\end{array}$$
Remark now that \[2^{L_0}+\sum_{q=q_0}^{a_{l_0}-a_{l_0-1}-1}\alpha_{a_{l_0-1}+q}2^{a_{l_0-1}+q}+\sum_{j=l_0}^{l_N-1}\sum_{q=0}^{a_{j+1}-a_j-1}\alpha_{a_{j}+q}2^{a_j+q}+\sum_{q=0}^{q_N}\alpha_{a_{l_N}+q}2^{a_{l_N}+q}=N+1\] and replace this in the preceding expression of $n_N$ to obtain:
\begin{align*}
n_N=&\ 2\left(\sum_{i=1}^{\infty}\frac{1}{2^{a_i-1}}\right) N +2\left(\sum_{i=1}^{\infty}\frac{1}{2^{a_i-1}}\right)+2\left(\sum_{j=l_0}^{1+l_N}\alpha_{a_{j}-1}\right)-2^{1+L_0}\left(\sum_{i=l_0-1}^{\infty}\frac{1}{2^{a_i-1}}\right)\\
&-2\left(\sum_{i=l_0}^{\infty}\frac{1}{2^{a_i-1}}\right)\left(\sum_{q=q_0}^{a_{l_0}-a_{l_0-1}-1}\alpha_{a_{l_0-1}+q}2^{a_{l_0-1}+q}\right)\\
&-2 \left(\sum_{j=l_0}^{l_N-1}\left(\sum_{i=j+1}^{\infty}\frac{1}{2^{a_i-1}}\right)\sum_{q=0}^{a_{j+1}-a_j-1}\alpha_{a_{j}+q}2^{a_j+q}\right)\\
&-2\left(\sum_{i=l_N+1}^{\infty}\frac{1}{2^{a_i-1}}\right)\left(\sum_{q=0}^{q_N}\alpha_{a_{l_N}+q}2^{a_{l_N}+q}\right)-l_0+1
\end{align*}
\end{proof}

Now observe that we have the following estimates:

\[
2\left(\sum_{i=l_0}^{\infty}\frac{1}{2^{a_i-1}}\right)\left(\sum_{q=q_0}^{a_{l_0}-a_{l_0-1}-1}\alpha_{a_{l_0-1}+q}2^{a_{l_0-1}+q}\right)\leq
4\left(\sum_{j=1}^{a_{l_0}-a_{l_0-1}-q_0}2^{-j}\right) +\frac{4}{2^{a_{l_0+1}-1}}2^{a_{l_0}}
\leq 8,
\]
\[
\begin{array}{rcl}\displaystyle 2 \left(\sum_{j=l_0}^{l_N-1}\left(\sum_{i=j+1}^{\infty}\frac{1}{2^{a_i-1}}\right)
\sum_{q=0}^{a_{j+1}-a_j-1}\alpha_{a_{j}+q}2^{a_j+q}\right)&
\leq& \displaystyle 8\left(\sum_{j=l_0}^{l_N-1} 2^{-(a_{j+1}-a_j)}\left(\sum_{q=0}^{a_{j+1}-a_j-1}2^q\right)\right)\\&&+
\displaystyle 16\left(\sum_{j=l_0}^{l_N-1} 2^{-(a_{j+2}-a_j)}\left(\sum_{q=0}^{a_{j+1}-a_j-1}2^q\right)\right)\\&\leq&
8(l_N-l_0)+16,\end{array}
\]
and
\[
2\left(\sum_{i=l_N+1}^{\infty}\frac{1}{2^{a_i-1}}\right)\left(\sum_{q=0}^{q_N}\alpha_{a_{l_N}+q}2^{a_{l_N}+q}\right)
\leq 4.
\]

The combination of these estimates with Lemma \ref{lem6} leads to the following statement. 

\begin{lemma}\label{lemma_lN} 
Using the notations of Definition \ref{defdecompoN}, there exist positive real numbers $C_1,C_2,C_3,C_4$ such that we have, for every positive integer $N,$ 
\[
2N\left(\sum_{i=1}^{\infty}\frac{1}{2^{a_i-1}}\right)
-C_1 l_N-C_2\leq n_N\leq 2N\left(\sum_{i=1}^{\infty}\frac{1}{2^{a_i-1}}\right)+C_3 l_N+C_4.
\]
\end{lemma}

Thus assume that $a_m=2^{2^{\adots^{2^m}}},$ where 
$2$ appears $s-1$ times ($s\geq 2$). Then the associated function $f:\mathbb{N}\rightarrow \mathbb{N}$ is given by 
$f(j)=m,$ for $j\in\{a_m,\dots,a_{m+1}-1\},$ and the following estimate holds:
\[
C N-C_1 \log^{(s)}(N)\leq n_N\leq C N+C_2 \log^{(s)}(N)
\]
with $C,C_1,C_2>0.$ An easy adaptation of Lemma 4.10 of \cite{ErMo} gives $\underline{d}_{\tilde{D}_s}((n_k))>0$ and we have proved 
the following result.

\begin{theorem}\label{main_thm_dens} Let $X$ be a Fr\'echet space, $Y$ a separable Fr\'echet space and 
$T_n:X\rightarrow Y,$ $n\in\mathbb{N},$ continuous mappings. If 
the sequence $(T_n)$ satisfies the frequent universality criterion, 
then $(T_n)$ is $\tilde{D}_s$-frequently universal for any $s\geq 1.$   
\end{theorem}

\begin{remark} \label{rem_fond}
\begin{enumerate}
\item \label{rem_fond1} Let $h$ be any real strictly increasing $C^1$-function satisfying the following conditions: $h(x)\rightarrow +\infty,$ as $x\rightarrow +\infty,$ 
$x\mapsto x/h(\log(x))$ (eventually) increases (to infinity), $h(x)=o(\log (x))$ and $h'(x)=o(h(x))$ as $x\rightarrow +\infty$. 
With a good choice of the sequence $(a_n)$ (i.e. $a_n=h^{-1}(n)$) one can show that 
the conditions of the frequent universality criterion automatically imply 
the $A_h$-frequent universality, where $A_h$ is the admissible matrix given by the coefficients $e^{k/h(\log(k))}.$ Indeed, under the previous assumptions, 
Lemma \ref{lemma_lN} ensures that there exist $C,C_1,C_2>0$ such that
\[
C N-C_1 h(\log(N))\leq n_N\leq C N+C_2 h(\log(N)).
\] 
Taking into account Lemma \ref{LemmaDensInfCalc} and the estimate $\sum_{k=1}^n e^{k/h(\log(k))}\sim h(\log(n))e^{n/h(\log(n))}$ as $n\rightarrow +\infty$, an adaptation of Lemma 4.10 of \cite{ErMo} gives $\underline{d}_{A_h}((n_k))>0$ and we have proved 
the result. 
\item Sophie Grivaux showed that one can find a proof of Theorem \ref{main_thm_dens} (or Remark \ref{rem_fond}(\ref{rem_fond1})) in the particular case of hypercyclicity in an elegant way 
thanks to ergodic theory arguments \cite{grivaux}. To do this, she combines Theorem 7 of \cite{JaSa} with Theorem 1 of \cite{MurPer}.  
\end{enumerate}
\end{remark}

We end this section by proving that an operator $T:X\rightarrow X,$ acting on a Fr\'echet space $X$, cannot be $A_1$-frequently hypercyclic. In other words, 
(\ref{rem_fond1}) from Remark \ref{rem_fond} ensures that we have obtained the sharpest result in the context of hypercyclicity. 

\begin{proposition}
Let $X$ be a Fr\'echet space. Then, there is no $A_1$-frequently hypercyclic operator on $X$.
\end{proposition}

\begin{proof} Let $T:X\rightarrow X$ be a continuous operator. Assume that the topology on $X$ is given by an increasing sequence of semi-norms $(p_j)$. 
Let us consider $J\geq 1$, $f\in X$ with $p_{J}(f)\geq 1$. 
We set $0<\varepsilon<1$ and $B_J(f,\varepsilon):=\{z\in F: p_J(z-f)<\varepsilon\}$. Assume that $T$ is $A_1$-frequently hypercyclic and let 
us consider $x\in X$ so that $x$ is a $A_1$-frequently hypercyclic vector for $T$. According to the proof of \cite[Proposition 3.7]{ErMo} 
the set $N(x,B_J(f,\varepsilon))$ has bounded gaps. Hence we denote by $M$ an upper bound for the length of these gaps. Let us define 
$K_0=J$ and $\eta_0=1-\varepsilon>0$. The continuity of the operator $T$ ensures that 
there exists a sequence of natural numbers $(K_i)_{1\leq i\leq M}$ and a sequence of positive numbers $(\eta_i)_{1\leq i\leq M}$ such that:
\begin{enumerate}
	\item for $0\leq i\leq M$, $K_i\geq J$,\label{P1}
	\item for $0\leq i\leq M$, $\eta_i\leq 1-\varepsilon$,\label{P2}
	\item for $1\leq i\leq M$ and for any $z\in X$,
	\[p_{K_i}(z)\leq\eta_i\implies p_{K_{i-1}}(T(z))\leq \eta_{i-1}.\]\label{P3}
\end{enumerate}
Now we use the $A_1$-frequent hypercyclicity of $x$ to find a natural number $n$ such that
\[
p_{K_M}(T^n(x))<\eta_M.
\]
It follows from (\ref{P3}) that for every $0\leq i\leq M$, 
\[p_{K_{M-i}}(T^{n+i}(x))<\eta_{M-i}.\]
Therefore, by (\ref{P1}) and (\ref{P2}), we get, for every $0\leq i\leq M$, 
\[p_{K_{J}}(T^{n+i}(x))<1-\varepsilon.\]
Thus, for every $0\leq i\leq M$, 
\[p_{J}(T^{n+i}(x)-f)\geq \vert p_J(f)-p_{J}(T^{n+i}(x))\vert\geq 1-(1-\varepsilon)=\varepsilon.\]
We have proved that $[n;n+M]\cap N(x,B_J(f,\varepsilon))=\emptyset$ which contradicts the definition of $M$. 
\end{proof}

\section{A log-frequently hypercyclic operator which is not frequently hypercyclic}\label{fhc_op} 
In their very nice paper, Bayart and Ruzsa \cite{Bayru} gave a characterization of frequently hypercyclic weighted 
shifts on the sequence spaces $\ell^p$ and $c_0$. In particular, a straightforward modification of the proof of \cite[Theorem 13]{Bayru}
gives an analogous result with respect to the so-called logarithmic density.

\begin{theorem}\label{Theologdensity}
	Let $w=(\omega_n)_{n\in\N}$ be a bounded sequence of positive integers. Then $B_w$ is log-frequently hypercyclic on $c_0(\N)$ if and only if there exist a sequence $(M(p))$ of positive real numbers tending to $+\infty$ and a sequence $(E_p)$ of subsets of $\N$ such that:
	\begin{enumerate}[(a)]
		\item For any $p\geq 1$, $\underline{d}_{log}\left(E_p\right)>0$;
		\item For any $p,q\geq1$, $p\neq q$, $\left(E_p+[0,p]\right)\cap\left(E_q+[0,q]\right)=\emptyset$;
		\item $\lim_{n\to\infty,\ n\in E_p+[0,p]} \omega_1\cdots \omega_n=+\infty$;
		\item For any $p,q\geq1$, for any $n\in E_p$ and any $m\in E_q$ with $m>n$, for any $t\in\{0,\ldots,q\}$,
		$$\omega_1\cdots \omega_{m-n+t}\geq M(p)M(q).$$
	\end{enumerate}
\end{theorem}

In the same paper, the authors also provide examples of a $\mathcal{U}$-frequently hypercyclic weighted shift which is not frequently hypercyclic and of a frequently hypercyclic weighted shift which is not distributionally chaotic. In what follows, we modify these constructions as well as those made in \cite{ErMo} 
to construct a log-frequently hypercyclic operator which is not frequently hypercyclic. We begin by the following lemma.

\begin{lemma}\label{lem_seg} There exist 
	$a>1$ and $\varepsilon>0$ such that for any integer $u>v\geq  1$, if we let $I_u^{a,\varepsilon}=[2^{(1-\varepsilon)a^{2u}},2^{(1+\varepsilon)a^{2u}}]$, then the following properties hold:
	\begin{enumerate}
		\item $I_{u}^{a,4\varepsilon}\cap I_{v}^{a,4\varepsilon}=\emptyset$\label{cond}
		\item $I_{u}^{a,2\varepsilon}-I_{v}^{a,2\varepsilon}\subset I_{u}^{a,4\varepsilon}$\label{cond0}
		\item $a^2\frac{1-4\varepsilon}{1+4\varepsilon}>1$\label{cond1}
	\end{enumerate}
\end{lemma}

\begin{proof}

A simple calculation shows that for any  $u>v\geq1$ then $I_{u}^{a,4\varepsilon}\cap I_{v}^{a,4\varepsilon}=\emptyset$ if and only if $a^2\left(\frac{1-4\varepsilon}{1+4\varepsilon}\right)\geq1$. Similarly, we have 
$$I_{u}^{a,\varepsilon}-I_{v}^{a,\varepsilon}\subset I_{u}^{a,4\varepsilon}\hbox{ if and only if }2^{a^{2(u-1)}(a^2(1-2\varepsilon)-(1+2\varepsilon)}\left(1-2^{2\varepsilon(2-a^2) a^{2(u-1)}}\right)\geq 1.$$ 
This can be achieved if $a^2(1-2\varepsilon)-(1+2\varepsilon)\geq1$ and $a^2\geq 2+\frac{1}{2\varepsilon}$ for example. Remark that these conditions and (\ref{cond1}) are satisfied if $\varepsilon$ is chosen to be very small and $a$ very large.	
\end{proof}

The philosophy of the Lemma \ref{lem_seg} can be summarized as follows: 
it suffices to choose $a$ very large and at the same time $\varepsilon$ very small to obtain the result stated. 
From now on, we suppose that $a$ and $\varepsilon$ are given by the previous lemma. 
Let $(A_p)$ be any syndetic partition of $\N$ and $M_p$ be the maximum length of the gaps in $A_p$.
We consider an increasing sequence of integers $(b_p)$ such that
\begin{equation}\label{cond2}
[b_p-8p;b_p+4p]\subset\bigcup_{u\geq2}\left]2^{(1-\varepsilon)a^{2(u-1)}};2^{(1-\varepsilon)a^{2u}}-2^{(1+\varepsilon)a^{2(u-1)}}\right[
\end{equation}
and
\begin{equation}\label{cond2bis}
b_p\geq (8p+1)2^p.
\end{equation}
This construction is possible if $a$ is chosen big enough.
Finally, let
\[E_p=\cup_{u\in A_p}\left(I_{u}^{a,\varepsilon}\cap b_p\N\right).\]

\begin{lemma} Under the above notations, we have $\underline{d}_{log}(E_p)>0$.	
\end{lemma}

\begin{proof} Let $(n_k)_{k\in\N}$ be an increasing enumeration of $A_p$. We set $[c_k;d_k]=I_{n_k}^{a,\varepsilon}$. Then, we get
\[
\underline{d}_{log}(E_p)= \liminf_{k\to\infty}\left(\frac{\sum_{l=1}^{k}\sum_{j=c_{l}\atop j\in b_p \N}^{d_{l}}\frac{1}{j}}{\sum_{j=1}^{c_{k+1}}\frac{1}{j}}\right)
\geq \liminf_{k\to\infty}\left(\frac{\sum_{j=c_{k}\atop j\in b_p \N}^{d_{k}}\frac{1}{j}}{\log\left(c_{k+1}\right)} \right)
\]
Since there exists $0\leq \alpha_k,\beta_k<b_p$ such that $[c_k+\alpha_k;d_k-\beta_k]\subseteq[c_k;d_k]$, $c_k+\alpha_k=\alpha b_p$ and
$d_k-\beta_k=\beta b_p$ for some $\alpha,\beta\in\N$, we obtain
\begin{align*}
\underline{d}_{log}(E_p)&\geq \liminf_{k\to\infty}\frac{\sum_{j=0}^{\beta-\alpha}\frac{1}{\alpha b_p+jb_p}}{\log\left(c_{k+1}\right)}\\
&\geq \liminf_{k\to\infty}\frac{\log\left(\frac{\beta}{\alpha}\right)}{b_p\log\left(c_{k+1}\right)}\\
&\geq\liminf_{k\to\infty}\frac{\log\left(\frac{d_k-\beta_k}{c_k+\alpha_k}\right)}{b_p\log\left(c_{k+1}\right)}\\
&\geq\liminf_{k\to\infty}\frac{\log\left(\frac{2^{(1+\varepsilon)a^{2n_k}}-b_p}{2^{(1-\varepsilon)a^{2n_k}}+b_p}\right)}{b_p\log\left(2^{(1-\varepsilon)a^{2n_{k+1}}}\right)}\\
&\geq\liminf_{k\to\infty}\frac{2\varepsilon a^{2n_k}}{b_p(1-\varepsilon)a^{2n_{k+1}}}\\
&\geq \frac{2\varepsilon}{b_p(1-\varepsilon)a^{2M_p}}>0.
\end{align*}
\end{proof}

Further, the following lemma is almost the same as \cite[Lemma 3]{Bayru} and it still holds in our context:

\begin{lemma}
	Let $p,q\geq1$, $n\in E_p$, $m\in E_q$ with $n\neq m$. Then $\vert n-m\vert>\max(p,q)$.
\end{lemma}

In particular, $(E_p+[0,p])\cap(E_q+[0,q])=\emptyset$ if $p\neq q$. Thus, the sequence of sets $(E_p)$ satisfies conditions $(a)$ and $(b)$ from Theorem \ref{Theologdensity}. Observe also that if we remove a finite number of elements from $A_p$ we may suppose that $A_p\subset[p,\infty[$ and for every $u\in A_p$, $I_{u}^{a,\varepsilon}+[-2p,2p]\subset I_{u}^{a,2\varepsilon}$.

We now turn to the construction of the weights of the weighted shift we are looking for. For this construction, we also draw our inspiration from constructions made in \cite{Bayru}. We set:
$$w_{1}^{p}\cdots w_{k}^{p}=\begin{cases}1 &\text{ if }k\notin b_p\N+[-4p,4p]\\
2^p &\text{ if }k\in b_p\N+[-2p,2p]
\end{cases}$$
and for every $k\in\N$, $\frac{1}{2}\leq w_{k}^{p}\leq 2$.
Then for $p,q\geq1$, $u\in A_p$ and $v\in A_q$ with  $u>v$ we define
$$w_{1}^{u,v}\cdots w_{k}^{u,v}\begin{cases}=1 &\text{ if }k\notin I_{u}^{a,4\varepsilon}\\
\geq\max(2^p,2^q) &\text{ if }k\in I_{u}^{a,\varepsilon}-I_{v}^{a,\varepsilon}+[0,p]\\
= 2^u &\text{ if }k\in I_{u}^{a,\varepsilon}+[0,p]
\end{cases}$$
with $\frac{1}{2}\leq w_{k}^{u,v}\leq 2$.
We must stress that this construction is made possible by the following conditions:
\begin{itemize}
	\item $I_{u}^{a,\varepsilon}-I_{v}^{a,\varepsilon}+[0,p]-[0,\max(p,q)]\subset I_{u}^{a,2\varepsilon}-I_{v}^{a,2\varepsilon}\subset I_{u}^{a,4\varepsilon}$ which ensures that we can pass from $1$ to $\max(2^p,2^q)$,
	\item $I_{u}^{a,\varepsilon}-[0,u]\subset I_{u}^{a,4\varepsilon}$ i.e. $2^{(1-4\varepsilon)a^{2u}}+u\leq 2^{(1-\varepsilon)a^{2u}}$ which is possible if $a$ is sufficiently big and $\varepsilon$ is sufficiently small. Thus, it is possible to grow from $1$ to $2^{u}$,
	\item $I_{u}^{a,\varepsilon}+[0,p]+[0,u]\subset I_{u}^{a,\varepsilon}+[0,2u] \subset I_{u}^{a,4\varepsilon}$ which is again possible if $a$ is sufficiently big and $\varepsilon$ is sufficiently small. This allows the weights to decrease from $2^u$ to $1$.
\end{itemize}

We are now able to give the definition of the weight $w$. This one is constructed in order to satisfy the following equality:
$$w_{1}\cdots w_{n}=\max_{p,u,v}\left(w_{1}^{p}\cdots w_{n}^{p},w_{1}^{u,v}\cdots w_{n}^{u,v}\right).$$
It is clear by construction that for every $n\in\N$, $\frac{1}{2}\leq w_n\leq2$, so the weighted backward shift $B_w$ is bounded and invertible. Moreover this construction satisfies condition $(c)$ in Theorem \ref{Theologdensity}.

Since we want to prove that $B_w$ is log-frequently hypercyclic, the only condition left to prove is condition $(d)$ from Theorem \ref{Theologdensity}.
Thus let $p,q\geq1$, $n\in E_p$ and $m\in E_q$ with $m>n$ and $t\in[0,q]$. Then we have two cases:
\medskip

$\bullet$ If $p=q$, then $m-n+t\in b_q\N+[0,q]$ and the definition of $w$ ensures that $w_{1}\cdots w_{m-n+t}\geq 2^{q}$,

\medskip
$\bullet$ If $p\neq q$, then there exists $u>v$ such that $n\in I_{v}^{a,\varepsilon}$ and $m\in I_{u}^{a,\varepsilon}$. Thus, by definition of $w$,
$$w_{1}\cdots w_{m-n+t}\geq \max\left(2^p;2^q\right)\geq 2^{\frac{p+q}{2}}\geq \lfloor2^{\frac{p}{2}}\rfloor\cdot\lfloor2^{\frac{q}{2}}\rfloor.$$

Now, one may define $M(p):=\lfloor2^{\frac{p}{2}}\rfloor$ and each case above satisfies condition $(d)$ from Theorem \ref{Theologdensity}. Thus we have proved that the weighted shift $B_w$ is logarithmically-frequently hypercyclic.\\

We now turn to the frequent hypercyclicity of $B_w$. We are going to prove by contradiction that $B_w$ is not frequently hypercyclic. 
Let us suppose that $B_w$ is frequently hypercyclic. Let also $x$ be a frequent hypercyclic vector and $E=\{n\in\N: \Vert B_{w}^{n}(x)-e_0\Vert\leq \frac{1}{2}\}$. 
Thus we have $\underline{d}(E)>0$ and $\lim_{n\to\infty, n\in E} w_1\cdots w_n=+\infty$. For every $p\geq 1$, we consider the set:
$$F_p=\{n\in E: w_1\cdots w_n>2^p\}.$$
This set is a cofinite subset of $E$, so it has the same lower density.
We also consider an increasing enumeration $(n_k)$ of $A_p$. For readability reasons, we define some notations here for the end of this part.
Let $[c_{k,\varepsilon};d_{k,\varepsilon}]=I_{n_k}^{a,\varepsilon}$. Then we get 
\[
\underline{d}(F_p)\leq\liminf_{k\to\infty}\left(
\displaystyle\frac{\#\{n\in F_p :n\leq c_{k+1,\varepsilon}\}}{c_{k+1,\varepsilon}}\right)
\]
We write
\[
\begin{array}{rcl}
\#\left\{n\in F_p :n\leq c_{k+1,\varepsilon}\right\}&\leq&
\#\left\{n\in F_p :n\leq d_{k,\varepsilon}\right\}\\&&+
\#\left\{n\in  \cup_{q>p}(b_q\N+[-4q;4q]) : d_{k,\varepsilon}<n\leq c_{k+1,\varepsilon}\right\}\\&&+
\#\left\{n\in \cup_{n_k\leq u\leq n_{k+1}}I_{u}^{a,4\varepsilon} :d_{k,\varepsilon}<n\leq c_{k+1,\varepsilon}\right\}.
\end{array}
\]

For the first term, remark that:
\begin{align*}
\frac{\displaystyle\#\left\{n\in F_p :n\leq d_{k,\varepsilon}\right\}}{c_{k+1,\varepsilon}}&\leq\frac{d_{k,\varepsilon}}{c_{k+1,\varepsilon}}\\
&\leq 2^{(1-\varepsilon)a^{2n_{k+1}}(\frac{1+\varepsilon}{1-\varepsilon}a^{2(n_k-n_{k+1})}-1)}\\
&\leq 2^{(1-\varepsilon)a^{2n_{k+1}}(a^{-2}\frac{1+\varepsilon}{1-\varepsilon}-1)}.
\end{align*}
Moreover, we can easily check that  $a^{-2}\frac{1+\varepsilon}{1-\varepsilon}-1<0$ by (\ref{cond1}). Thus we deduce that this first sum tends to zero.
Let us estimate the third term $J_{k,\varepsilon}^{(3)}:= 
\frac{\displaystyle\#\left\{n\in\cup_{n_k\leq u\leq n_{k+1}}I_{u}^{a,4\varepsilon} :d_{k,\varepsilon}<n\leq c_{k+1,\varepsilon}\right\}}{c_{k+1,\varepsilon}}$. 
We have

\begin{align*}
J_{k,\varepsilon}^{(3)}&\leq \frac{\sum_{u=n_k}^{n_{k+1}}\sum_{j=2^{(1-4\varepsilon)a^{2u}}}^{2^{(1+4\varepsilon)a^{2u}}}1}{c_{k+1,\varepsilon}}\\
&\leq \sum_{u=n_k}^{n_{k+1}}\frac{2^{(1+4\varepsilon)a^{2u}}-2^{(1-4\varepsilon)a^{2u}}+1}{2^{(1-\varepsilon)a^{2n_{k+1}}}}\\
&\leq\sum_{u=n_k}^{n_{k+1}} 2^{(1-\varepsilon)a^{2n_{k+1}}(\frac{1+4\varepsilon}{1-\varepsilon}a^{2(u-n_{k+1})}-1)}\left(1-2^{-8\varepsilon a^{2u}}+2^{-(1+4\varepsilon)a^{2u}}\right)\\
&\leq 2\sum_{u=n_k}^{n_{k+1}} 2^{(1-\varepsilon)a^{2n_{k+1}}(a^{-2}\frac{1+4\varepsilon}{1-4\varepsilon}-1)}\\
&\leq 2(1+M_p) 2^{(1-\varepsilon)a^{2n_{k+1}}(a^{-2}\frac{1+4\varepsilon}{1-4\varepsilon}-1)}\rightarrow 0.
\end{align*}

We now focus on the second term 
$J_{k,\varepsilon}^{(2)}:=\frac{\#\left\{n\in  \cup_{q>p}(b_q\N+[-4q;4q]) : d_{k,\varepsilon}<n\leq c_{k+1,\varepsilon}\right\}}{c_{k+1,\varepsilon}}$. 
Notice that the union over $q>p$ is in fact a finite union. For those $q>p$ such that $c_{k+1,\varepsilon}+4q-(d_{k,\varepsilon}-4q)\geq b_q$ we get 
\begin{align*}
\#\left\{ ]d_{k,\varepsilon}; c_{k+1,\varepsilon}]\cap \left(b_q\N+[-4q;4q]\right)\right\}&\leq(8q+1)\#\left\{ ]2^{(1+\varepsilon)a^{2n_{k}}}-4q; 2^{(1-\varepsilon)a^{2n_{k+1}}}+4q]\cap b_q\N\right\}\\
&\leq 3(8q+1)\frac{2^{(1-\varepsilon)a^{2n_{k+1}}}-2^{(1+\varepsilon)a^{2n_{k}}}+8q}{b_q}
\end{align*}

Moreover, by condition (\ref{cond2}), there is no $q>p$ such that $c_{k+1,\varepsilon}-d_{k,\varepsilon}+8q< b_q<c_{k+1,\varepsilon}+4q$. 
Thus, replacing the finite sum by an infinite one we obtain that
\begin{align*}
J_{k,\varepsilon}^{(2)}&\leq \sum_{q>p}\frac{3(8q+1)}{2^{(1-\varepsilon)a^{2n_{k+1}}}}\frac{2^{(1-\varepsilon)a^{2n_{k+1}}}-2^{(1+\varepsilon)a^{2n_{k}}}+8q}{b_q}\\
&\leq \sum_{q>p}3(8q+1)\frac{1-2^{(1+\varepsilon)a^{2n_{k}}-(1-\varepsilon)a^{2n_{k+1}}}+8q}{b_q}\\
&\leq 3\sum_{q>p}(8q+1)\frac{1-2^{(1-\varepsilon)a^{2n_{k+1}}}(\frac{1+\varepsilon}{1-\varepsilon}a^{2(n_k-n_{k+1})}-1)+8q}{b_q}\\
&\leq
6\sum_{q>p}\frac{8q+1}{b_q}.
\end{align*}
Since this last inequality does not require any property on $p$, we can let $p$ tend to infinity which, thanks to (\ref{cond2bis}), implies that $\underline{d}(E)=\lim_{p\to\infty} \underline{d}(F_p)=0$, hence we obtain a contradiction.
Thus the weighted shift $B_w$ is not frequently hypercyclic. From this construction, we deduce the following 
result.

\begin{theorem}\label{counterexample} There exists a log-frequently hypercyclic operator being not frequently hypercyclic.
\end{theorem}

\vskip3mm

\noindent {\bf Acknowledgements.} The authors were partly supported by the 
grant ANR-17-CE40-0021 of the French
National Research Agency ANR (project Front).

\end{document}